\documentclass[a4paper]{article}

\usepackage{amsmath}
\usepackage{amssymb}
\usepackage{amsthm}
\usepackage{bm}
\usepackage{dsfont}
\usepackage{enumerate}
\usepackage{ifthen}
\usepackage{pifont}
\usepackage{rotating}
\usepackage{stmaryrd}
\usepackage{theoremref}
\usepackage{textcomp}
\input xy
\xyoption{all}

\newcommand{\mc}[1]{{\mathcal{#1}}}
\newcommand{\mf}[1]{{\mathfrak{#1}}}
\newcommand{\bb}[1]{{\mathbb{#1}}}

\DeclareMathOperator{\RE}{Re}
\DeclareMathOperator{\IM}{Im}
\renewcommand{\Re}{\RE}
\renewcommand{\Im}{\IM}

\newcommand{\qu}{\overline}
\DeclareMathOperator{\Clos}{Clos}
\DeclareMathOperator{\id}{id}
\DeclareMathOperator{\mt}{mt}
\DeclareMathOperator{\ran}{ran}
\DeclareMathOperator{\spn}{span}
\DeclareMathOperator{\supp}{supp}
\newlength{\maxlabwidth}
\newenvironment{axioms}[1]{
	\setlength{\maxlabwidth}{#1}
	\begin{list}{}{
	\setlength{\rightmargin}{2mm}
	\setlength{\leftmargin}{\maxlabwidth}\addtolength{\leftmargin}{2mm}
	\setlength{\labelsep}{0mm}
	\setlength{\labelwidth}{\maxlabwidth}
	\setlength{\itemindent}{0mm}
	
	}
	}{
	\end{list}
	}

\numberwithin{equation}{section}
\swapnumbers
\theoremstyle{plain}
	\newtheorem{lemma}{Lemma}[section]
	
	\newtheorem{theorem}[lemma]{Theorem}
	\newtheorem{corollary}[lemma]{Corollary}
	
\theoremstyle{definition}
	\newtheorem{definitioN}[lemma]{Definition}
\theoremstyle{remark}
	\newtheorem{remarK}[lemma]{Remark}
	\newtheorem{examplE}[lemma]{Example}
	\newtheorem{nremarK}[lemma]{}
\newcommand{\thlab}[1]{\thlabel{#1}\label{#1.}}
\renewcommand{\qedsymbol}{\raisebox{-2pt}{\large\ding{113}}}% Box mit Schatten
\newcommand{\defendsymbol}{$\sslash$}
\newcommand{\qedsymbolsave}{\qedsymbol}
\newenvironment{definition}{\begin{definitioN}}{
	\renewcommand{\qedsymbolsave}{\qedsymbol}\renewcommand{\qedsymbol}{\defendsymbol}
	\popQED{\qed}\renewcommand{\qedsymbol}{\qedsymbolsave}\end{definitioN}}
\newenvironment{remark}{\begin{remarK}}{
	\renewcommand{\qedsymbolsave}{\qedsymbol}\renewcommand{\qedsymbol}{\defendsymbol}
	\popQED{\qed}\renewcommand{\qedsymbol}{\qedsymbolsave}\end{remarK}}

\newenvironment{proofof}[1]{\begin{proof}[\textit{Proof (of #1)}]}{\end{proof}}
\newcommand{\bibi}[5]{\bibitem[#5]{#1} \textsc{#2}:\ \textit{#3,}\ {#4.}}

\begin{document}

%\renewcommand{\thefootnote}{\fnsymbol{footnote}}
%\footnotetext{\textcircledP\ Preliminary version Thu 19 Jul 2012 12:15
%}
%\renewcommand{\thefootnote}{\arabic{footnote}}\setcounter{footnote}{0}

{\Large\bf
\begin{flushleft}
	De Branges' theorem on approximation problems of Bernstein type
\end{flushleft}
}
\vspace*{3mm}
\begin{center}
	{\sc Anton Baranov, Harald Woracek}
\end{center}

\begin{abstract}
	\noindent
	The Bernstein approximation problem is to determine whether or not the space of all
	polynomials is dense in a given weighted $C_0$-space on the real line. A theorem of
	L.\,de~Branges characterizes non--density by existence of an entire function of
	Krein class being related with the weight in a certain way. An analogous result
	holds true for weighted sup--norm approximation by entire functions of exponential type
	at most $\tau$ and bounded on the real axis ($\tau>0$ fixed). 

	We consider approximation in weighted $C_0$-spaces by functions belonging to a
	prescribed subspace of entire functions which is solely assumed to be invariant under
	division of zeros and passing from $F(z)$ to $\qu{F(\qu z)}$, and establish the precise
	analogue of de~Branges' theorem. For the proof we follow the lines of de~Branges'
	original proof, and employ some results 
	of L.\,Pitt. 
\end{abstract}
\begin{flushleft}
	{\small
	}
\end{flushleft}

%---------
%   TEXTBODY
%---------

%
%
%
\section{Introduction}

Several classical problems revolve around the following general question: 
\begin{quote}
	\emph{Let $X$ be a Banach space of functions, and let $\mc L$ be a linear subspace of $X$. When is 
	$\mc L$ dense in $X$ ?}
\end{quote}
As a model example, let us discuss weighted sup-norm approximation by polynomials. 
Let $W:\bb R\to(0,\infty)$ be continuous, and assume that 
$\lim_{|x|\to\infty}\frac{x^n}{W(x)}=0$, $n\in\bb N$. 
Take for $X$ the space $C_0(W)$ of all continuous functions on the real line such that 
$\frac{f}{W}$ tends to zero at infinity, endowed with the norm 
$\|f\|_{C_0(W)}:=\sup_{x\in\bb R}\frac{|f(x)|}{W(x)}$. And 
take for $\mc L$ the space $\bb C[z]$ of all polynomials with complex coefficients. 
Then the above quoted question is known as the Bernstein problem. Several answers were obtained in the 1950's, e.g., in the 
work of Mergelyan \cite{mergelyan:1956}, Akhiezer \cite{akhiezer:1956}, Pollard \cite{pollard:1953}, or de~Branges \cite{debranges:1959b}. 
A comprehensive exposition can be found in \cite[Chapter VI.A--D]{koosis:1988}, let us also 
refer to the survey article \cite{lubinsky:2007} where also quantitative results are reviewed. 

Most characterizations of density use in some way the function 
\[
	\mf m(z):=\sup\big\{|P(z)|:\,P\in\bb C[z],\|P\|_{C_0(W)}\leq 1\big\}
	\,,
\]
also referred to as the Hall majorant associated with the weight $W$. 
The value $\mf m(z)$ is of course nothing but the norm of the point evaluation functional $F\mapsto F(z)$ on $\bb C[z]$ with respect 
to the norm $\|.\|_{C_0(W)}$, understanding `$\mf m(z)=\infty$' as point evaluation being unbounded. 

\begin{theorem}\thlab{N4}
	Let $W:\bb R\to(0,\infty)$ be a continuous weight, and assume that $\lim_{|x|\to\infty}x^nW(x)^{-1}=0$, $n\in\bb N$. 
	Then the following are equivalent:
	\begin{enumerate}[$(i)$]
		\item $\bb C[z]$ is dense in $C_0(W)$. 
		\item {\rm(Mergelyan, 1956)} We have $\mf m(z)=\infty$ for one 
			$($equivalently, for all$)$ $z\in\bb C\setminus\bb R$. 
		\item {\rm(Akhiezer, 1956)} We have 
			\[
				\int_{\bb R}\frac{\log\mf m(x)}{1+x^2}\,dx=\infty
				\,.
			\]
		\item {\rm(Pollard, 1953)} We have 
			\[
				\sup\Big\{\int_{\bb R}\frac{\log|P(x)|}{1+x^2}\,dx:\,P\in\bb C[z],\|P\|_{C_0(W)}\leq 1\Big\}=\infty
				\,.
			\]
	\end{enumerate}
\end{theorem}

\noindent
The criterion of de~Branges is of different nature. 

\begin{theorem}[de~Branges, 1959]\thlab{N5}
	Let $W:\bb R\to(0,\infty)$ be a continuous weight, and assume that 
	$\lim_{|x|\to\infty}\frac{x^n}{W(x)}=0$, $n\in\bb N$. 
	Then the following are equivalent:
	\begin{enumerate}[$(i)$]
		\item $\bb C[z]$ is not dense in $C_0(W)$. 
		\item There exists an entire function $B$ which possesses the properties 
			\begin{itemize}
			\item[--] $B$ is not a polynomial. We have $B(z)=\qu{B(\qu z)}$. 
				All zeros of $B$ are real and simple. 
			\item[--] $B$ is of finite exponential type, and 
				$\int_{\bb R}\frac{\log^+|B(x)|}{1+x^2}\,dx<\infty$. 
			\item[--] $\displaystyle \sum_{x:\,B(x)=0}\frac{W(x)}{|B'(x)|}<\infty$. 
			\end{itemize}
	\end{enumerate}
	If $\bb C[z]$ is not dense in $C_0(W)$, the function $B$ in $(ii)$ can be chosen of
	zero exponential type. 
\end{theorem}

\noindent
De~Branges' original proof uses mainly basic functional analysis (the Krein--Milman
theorem) and complex analysis (bounded type theory). Several other approaches are known; for 
example \cite{sodin.yuditskii:1996} where the result is obtained as a consequence of a deep study
of Chebyshev sets, or \cite{poltoratski:2011} where some properties of singular (Cauchy-) 
integral operators are invoked.

Also other instances for $X$ and $\mc L$ in the question quoted in the very first paragraph 
of this introduction are classical objects of study. 
\\[1mm]
\textit{Working in other spaces $X$:}
Density of polynomials in a space $L^2(\mu)$ or $L^1(\mu)$ is closely related with the 
structure of the solution set of the Hamburger power moment problem generated by the sequence 
$(\int_{\bb R}x^n\,d\mu(x))_{n\in\bb N_0}$. In fact, by theorems of M.\,Riesz and 
M.A.\,Naimark, density is equivalent to extremal properties of $\mu$ in this solution set, cf.\ 
\cite[\S 2.3]{akhiezer:1961}. Also a characterization in terms of the
norms of point evaluation maps was obtained already at a very early stage (in the 1920's) by
M.\,Riesz, see, e.g., \cite[Chapter V.D]{koosis:1988}. By the recent 
work of A.G.\,Bakan the results for 
$L^p$-spaces and weighted $C_0$-spaces are closely related, see \cite{bakan:2008}. 
\\[1mm]
\textit{Approximation with functions different from polynomials:} 
For example, let us mention approximation with entire functions of finite exponential 
type. There the space $\mc L$ is taken, e.g., as the set of all finite linear
combinations of exponentials $e^{i\lambda x}$, $|\lambda|\leq a$, (with some fixed $a>0$) or 
as the space of all Fourier transforms of $C^\infty$-functions compactly supported in $(-a,a)$. 
Analogues of the mentioned theorems are again classical, see, e.g., 
\cite[Chapter VI.E--F]{koosis:1988}. A proof of the 'exponential version' of de~Branges' theorem 
following the method of \cite{sodin.yuditskii:1996} is given in \cite{borichev.sodin:2011}.
\\[2mm]
In the 1980's, L.\,Pitt proposed a unifying (and generalizing) approach to approximation problems
of this kind, cf.\ \cite{pitt:1983}. He considered quite general instances of $X$ and 
$\mc L$: the Banach space $X$ is only assumed to be a so-called regular function space, 
and the space $\mc L$ can be any space of entire functions which is closed with respect to 
forming difference quotients and passing from $F(z)$ to $\qu{F(\qu z)}$ and which is contained 
injectively in $X$. The class of regular function spaces is quite large; for example it includes 
spaces $C_0(W)$, or $L^p(\mu)$, $p\in[1,\infty)$, or weighted Sobolev spaces. 
Under some mild regularity conditions Pitt shows analogues of the results mentioned in 
\thref{N4} above, as well as versions of some more detailed results in the same flavour which 
we did not list above. A general version of \thref{N5}, however, is not given. 

%%%
%
%
\begin{flushleft}
	\textbf{The present contribution.}
\end{flushleft}
\vspace*{-2mm}
Our aim in the present paper is to prove a theorem of de~Branges type for weighted sup-norm 
approximation by entire functions of a space $\mc L$ as considered in the work of Pitt: 
Namely, the below \thref{N3}. To establish this theorem, we follow de~Branges' original method. 

Independently of the present work, M.\,Sodin and P.\,Yuditskii have generalized the 
method first used in \cite{sodin.yuditskii:1996}, and obtained precisely the same 
result\footnote{This is work in progress communicated to us 
by M.\,Sodin, who also presented it at the 'International Symposium on Orthogonal Polynomials 
and Special Functions -- a Complex Analytic Perspective'  (June 11--15, 2012, Copenhagen).}.

In order to concisely formulate the presently discussed general version of \thref{N5}, 
we introduce some notation. First, for completeness, the class of weighted spaces under consideration. 

\begin{definition}\thlab{N24}
	We call a function $W:\bb R\to(0,\infty]$ a weight, 
        if $W$ is lower semicontinuous 
	and not identically equal to $\infty$. 

	We denote by $C_0(W)$ the space of all continuous functions $f$ on the real line such that 
	$\lim_{|x|\to\infty} \frac{|f(x)|}{W(x)} =0$ (here $\frac{f(x)}{W(x)}$ is understood as $0$, if $W(x)=\infty$). 
	This linear space is endowed with the seminorm 
	\[
		\|f\|_{C_0(W)}:=\sup_{x\in\bb R}\frac{|f(x)|}{W(x)}
		\,.
	\]
\end{definition}

\noindent
Clearly, $\|.\|_{C_0(W)}$ induces a locally convex vector topology on $C_0(W)$ and all topological notions are understood 
with respect to it. Notice that this topology need 
not be Hausdorff. In fact, the seminorm $\|.\|_{C_0(W)}$ is a norm if and only if the set 
\[
	\Omega:=\big\{x\in\bb R:\,W(x)\neq\infty\big\}
\]
is dense in $\bb R$. 

Next, a terminology for the spaces $\mc L$ with which we deal. 

\begin{definition}\thlab{N1}
	Let $\mc L$ be a linear space. We call $\mc L$ an algebraic de~Branges space, if 
	\begin{axioms}{12mm}
	\item[B1] The elements of $\mc L$ are entire functions. 
	\item[B2] If $F\in\mc L$ and $w\in\bb C$ with $F(w)=0$, then also the function 
		$\frac{F(z)}{z-w}$ belongs to $\mc L$. 
	\item[B3] If $F\in\mc L$, then also the function $F^\#(z):=\qu{F(\qu z)}$ belongs to 
		$\mc L$. 
	\end{axioms}
\end{definition}

\noindent
The appropriate weighted analogue of the class of entire functions appearing in 
de~Branges' theorem, which is also referred to as the Krein-class, is the following. 

\begin{definition}\thlab{N2}
	Let $\mc L$ be an algebraic de~Branges space. Let $W:\bb R\to(0,\infty]$ be a lower 
	semicontinuous function, and assume that $\mc L\subseteq C_0(W)$. 
	Then we define the $W$-weighted Krein class $\mc K(\mc L,W)$ 
	with respect to $\mc L$ as the set of all entire function $B$ which satisfy 
	\begin{axioms}{12mm}
	\item[K1] The function $B$ is not zerofree, and all its zeros are real and simple. 
		We have $B=B^\#$. 
	\item[K2] For each $F\in\mc L$, the function $\frac FB$ is of bounded type in both 
		half-planes $\bb C^+$ and $\bb C^-$. 
	\item[K3] For each $F\in\mc L$ we have 
		\[
			|F(iy)|=o\big(|B(iy)|\big),\quad y\to\pm\infty
			\,.
		\]
	\item[K4] $\displaystyle \sum_{x:\,B(x)=0}\frac{W(x)}{|B'(x)|}<\infty$. 
		\footnote{Convergence of this series implicitly includes the requirement that $W(x)<\infty$
		whenever $x$ is a zero of $B$.}
	\end{axioms}
\end{definition}

\noindent
Finally, we denote $\mt h:=\limsup_{y\to+\infty}\frac 1y\log|h(iy)|\in[-\infty,\infty]$ whenever $h$ is a complex-valued function 
defined (at least) on the ray $i\bb R^+$, and refer to this number as the mean type of $h$\footnote{Concerning this notation, 
we do not assume that $h$ is subharmonic or even analytic.}. 

The statement we are going to prove in this paper can now be formulated as follows. 

\begin{theorem}\thlab{N3}
	Let $\mc L$ be an algebraic de~Branges space. Let $W:\bb R\to(0,\infty]$ be a weight, i.e.\ 
	lower semicontinuous and not identically equal to $\infty$, and assume that $\mc L\subseteq C_0(W)$. 
	Then the following are equivalent:
	\begin{enumerate}[$(i)$]
		\item $\mc L$ is not dense in $C_0(W)$.
		\item $\mc K(\mc L,W)$ is nonempty.
	\end{enumerate}
	If $\mc L$ is not dense in $C_0(W)$, then there exists a function 
	$B\in\mc K(\mc L,W)$ with 
	\begin{equation}\label{N27}
		\sup\Big\{\mt\frac FB:\,F\in\mc L\Big\}=0
		\,.
	\end{equation}
\end{theorem}

\section{The toolbox}

Following de~Branges' original proof requires several (crucial!) tools: 
\\[1mm]
(1) A description of the topological dual space of $C_0(W)$. 
Knowledge about duals of weighted $C_0$-spaces (actually, in a much more general 
setting than the present) was obtained already in the 1960's by W.H. Summers 
following the work of L. Nachbin, cf.\ \cite{summers:1969}, \cite{summers:1970}, 
\cite{nachbin:1965}. 
\\[1mm]
(2) A weighted version of 
de~Branges' lemma on extremal points $\mu$ of the annihilator of $\mc L$. The original version 
(without any weights) is \cite[Lemma]{debranges:1959a}\footnote{See also 
\cite[Chapter VI.F]{koosis:1988}, where a particular case is elaborated.}. 
\\[1mm]
(3) Some of L.\,Pitt's theorems, applied with the Banach space $L^1(\mu)$. These can be 
extracted from \cite{pitt:1983}.
\\[1mm]
\noindent
We start with the description of bounded linear functionals. 
Let $X$ be a locally compact Hausdorff space, and let $W:X\to(0,\infty]$ be a lower semicontinuous
function. Set $\Omega:=\{x\in X:\,W(x)\neq\infty\}$ and 
\begin{equation}\label{N22}
	V(x):=
	\begin{cases}
		\frac 1{W(x)}, &\hspace*{-3mm}\quad x\in\Omega,
		\\
		0, &\hspace*{-3mm}\quad x\in X\setminus\Omega,

	\end{cases}
	\,,
\end{equation}
and denote by $\bb M_b(X)$ the space of all complex (bounded) Borel measures on $X$ endowed 
with the norm $\|\mu\|:=|\mu|(X)$, where $|\mu|$ denotes the total variation of the complex 
measure $\mu$. 

Consider the map $T$ which assigns to each measure $\mu\in\bb M_b(X)$ the linear 
functional $T\mu$ defined as 
\begin{equation}\label{N21}
	(T\mu)f:=\int_X fV\,d\mu,\quad f\in C_0(W)
	\,.
\end{equation}
Obviously, $T$ is well-defined and maps $\bb M_b(X)$ into $C_0(W)'$, in fact 
\[
	\|T\mu\|\leq\|\mu\|,\quad \mu\in\bb M_b(X)
	\,.
\]
The following statement is an immediate consequence of \cite[Theorems 3.1 and 4.5]{summers:1970},
just using in addition some standard approximation arguments (like Lusin's Theorem, cf.\
\cite[2.24]{rudin:1987}); we will not go into the details. 

\begin{theorem}[Summers]\thlab{N20}
	Let $X\neq\emptyset$ be a locally compact Hausdorff space, let $W:X\to(0,\infty]$ be
	lower semicontinuous, and assume that $\Omega:=\{x\in X:\,W(x)\neq\infty\}$ is dense 
	in $X$. Then the map $T$ defined by \eqref{N21} maps $\bb M_b(X)$ surjectively onto 
	$C_0(W)'$. Moreover, for each $\mu\in\bb M_b(X)$, the following hold\footnote{We write 
	$h\mu$ for the measure which is absolutely continuous with respect to $\mu$ and has Radon--Nikodym 
	derivative $h$. Moreover, $\mathds{1}_Y$ denotes the characteristic function of the set $Y$.}.
	\begin{enumerate}[$(i)$]
	\item We have $T\mu=T(\mathds{1}_\Omega\mu)$ and $\|T\mu\|=\|\mathds{1}_\Omega\mu\|$.
	\item The functional $T\mu$ is real $($i.e.\ $\forall\,f\in C_0(X),f\geq 0:\ 
		(T\mu)f\in\bb R$$)$, if and only if $\mathds{1}_\Omega\mu$ is a real-valued 
		measure. 
	\end{enumerate}
\end{theorem}

\noindent
Next, we show the required weighted version of de~Branges' lemma. This is done in essence 
by repeating the proof given in \cite{debranges:1959a}; for completeness, we provide the details. 

\begin{lemma}[de~Branges]\thlab{N18}
	Let $W$ be a weight, and let $\mc L$ be a linear subspace 
	of $C_0(W)$ which is invariant with respect to complex conjugation. 
	Assume that $\qu{\mc L}\neq C_0(W)$. 
	Then there exists a positive Borel measure on $\bb R$, $\mu\neq 0$, such that the
	following hold.
	\begin{enumerate}[$(i)$]
	\item $\displaystyle \int_{\bb R}W\,d\mu<\infty$. 
	\item The annihilator of the space $\mc L=\{f:\,f\in\mc L\}\subseteq L^1(\mu)$ with 
		respect to the duality between $L^1(\mu)$ and $L^\infty(\mu)$ 
		is one-dimensional. It is spanned by a function $g_0\in L^\infty(\mu)$ with 
		$|g_0(x)|=1$, $\mu$-a.a.\ $x\in\bb R$.
	\end{enumerate}
\end{lemma}
\begin{proof}
	For each $\sigma\in\bb M_b(\qu\Omega)$ we define a positive Borel measure $\tilde\sigma$
	on $\bb R$ by (the function $V$ is again defined by \eqref{N22}) 
	\[
		\tilde\sigma(E):=\int_{E\cap\Omega}V\,d|\sigma|,\quad  E\text{ is a Borel set}\,.
	\]
	Note that $V$ is upper semicontinuous, and hence bounded on every compact subset of
	$\qu\Omega$. Thus $\tilde\sigma(E)<\infty$ whenever $E\subseteq\bb R$ is compact, and
	hence $\tilde\sigma$ indeed is a positive Borel measure on the real 
	line\footnote{We always include the requirement that the measure of compact sets is 
	finite into the notion of a Borel measure.}. Clearly, we have 
	$\tilde\sigma(\bb R\setminus\Omega)=0$ and $\int_{\bb R}W\,d\tilde\sigma<\infty$, in
	particular thus $C_0(W)\subseteq L^1(\tilde\sigma)$. 

	\hspace*{0pt}\\[-2mm]\textit{Step 1:} 
	Let $\sigma\in\bb M_b(\qu\Omega)$ with $T\sigma\neq 0$ be fixed. For each measurable and
	bounded function $g:\bb R\to\bb C$, we set ($T$ is defined as in \eqref{N21} using the
	weight $W|_{\qu\Omega}$) 
	\[
		\Gamma_\sigma g:=T(g \sigma)
		\,.
	\]
	Then 
	\begin{equation}\label{N23}
		(\Gamma_\sigma g)f=\int_{\qu\Omega}fV\cdot g\,d\sigma=
		\int_{\qu\Omega}f\cdot g\frac{d\sigma}{d|\sigma|}\,d\tilde\sigma,\quad 
		f\in C_0(W|_{\qu\Omega})
		\,.
	\end{equation}
	Denote by $C_{00}(\qu\Omega)$ the space of all continuous functions on $\qu\Omega$ which have compact support. 
	Then $C_{00}(\qu\Omega)\subseteq C_0(W|_{\qu\Omega})$, and hence \eqref{N23} implies that
	$\Gamma_\sigma g=0$ if and only if $g(x)=0$  for $\tilde\sigma$-a.a.\ $x\in\qu\Omega$.
	We conclude that $\Gamma_\sigma$ induces a well-defined and injective linear operator
	(again denoted as $\Gamma_\sigma$) 
	\[
		\Gamma_\sigma:L^\infty(\tilde\sigma)\to C_0(W|_{\qu\Omega})'
		\,.
	\]
	Using the properties of $T$, we see that 
	\begin{enumerate}[$(a)$]
	\item $\|\Gamma_\sigma\|\leq\|\sigma\|$;
	\item if $\sigma$ is real-valued, then $\Gamma_\sigma$ maps real-valued functions to real
		functionals. 
	\end{enumerate}
	Again invoking \eqref{N23} we see that (Caution! The two annihilators are understood with
	respect to different dualities) 
	\begin{enumerate}[$(a)$]
	\setcounter{enumi}{2}
	\item $\Gamma_\sigma g\in\mc L^\perp$ ($\subseteq C_0(W|_{\qu\Omega})'$) if and only if
		$g\in\big(\frac{d\sigma}{d|\sigma|}\big)^{-1}\mc L^\perp$ ($\subseteq
		L^\infty(\tilde\sigma)$).
	\end{enumerate}

	\hspace*{0pt}\\[-2mm]\textit{Step 2:}
	Consider the set 
	\[
		\Sigma:=\big\{\phi\in C_0(W|_{\qu\Omega})':\,\|\phi\|\leq 1,\ \phi\text{
		real}\big\}\cap \mc L^\perp
		\,.
	\]
	Clearly, $\Sigma$ is $w^*$-compact and convex. Since $\qu{\mc L}\neq C_0(W)$, also 
	$\qu{\mc L}\neq C_0(W|_{\qu\Omega})$. Hence, there exists $\phi\in C_0(W|_{\qu\Omega})'$,
	$\|\phi\|=1$, with $\phi(f|_{\qu\Omega})=0$, $f\in\mc L$. Let $\mu\in\bb M_b(\qu\Omega)$,
	$\|\mu\|=1$ be such that $T\mu=\phi$. Since $\mc L$ is invariant with respect to complex
	conjugation, the functional $T\qu\mu$ also annihilates $\mc L$. Since $T\mu\neq 0$, one
	of $T(\Re\mu)$ and $T(\Im\mu)$ must be nonzero. We conclude that $\Sigma$ contains a
	nonzero element. 

	Let $\sigma\in\bb M_b(\qu\Omega)$ with $|\sigma|(\qu\Omega\setminus\Omega)=0$ and 
	$T\sigma\in\Sigma\setminus\{0\}$, and set 
	\[
		\mc M_\sigma:=\Big\{g\in\Big(\frac{d\sigma}{d|\sigma|}\Big)^{-1}\mc L^\perp:\,
		g\geq 0,\ \int_{\qu\Omega}g\,d|\sigma|=\|\sigma\|\Big\}
		\,.
	\]
	Due to the properties $(a)$--$(c)$ from above, we see that 
	\begin{enumerate}[$(A)$]
	\item $1\in\mc M_\sigma$;
	\item $\Gamma_\sigma(\mc M_\sigma)\subseteq\Sigma$.
	\end{enumerate}
	\begin{quote}
	\textit{Claim: If $\dim\mc L^\perp>1$ ($\mc L^\perp\subseteq L^\infty(\tilde\sigma)$),
	then $1$ is not an extremal point of $\mc M_\sigma$.}
	\end{quote}
	Once this claim is established, the assertion of the lemma follows immediately: By the
	Krein-Milman theorem, the set $\Sigma$ must contain a nonzero extremal point $\phi_0$.
	Let $\sigma_0\in\bb M_b(\qu\Omega)$ be such that 
        $|\sigma_0|(\qu\Omega\setminus\Omega)=0$
	and $T\sigma_0=\phi_0$. Then the function $1$ must be
	an extremal point of $\mc M_{\sigma_0}$ (otherwise, by the property (B)
        and linearity of $\Gamma_{\sigma_0}$, the function $\phi_0=T\sigma_0=
        \Gamma_{\sigma_0}1$ will not be an extremal point of $\Sigma$). 
        Hence, the measure $\mu:=\tilde\sigma_0$ has all
	properties required in the assertion of the lemma. 

	\hspace*{0pt}\\[-2mm]\textit{Step 3; Proving the claim:} 
	The measure $\sigma$ is real-valued. This implies that 
	$\big(\frac{d\sigma}{d|\sigma|}\big)^{-1}\mc L^\perp$ is invariant under complex
	conjugation, and hence that it equals the linear span of its real-valued elements. If 
	$\dim\mc L^\perp>1$, therefore, there must exist a real-valued element 
	$g\in\big(\frac{d\sigma}{d|\sigma|}\big)^{-1}\mc L^\perp$ which is not equal to a
	constant $\tilde\sigma$-a.e. Set 
	\[
		h:=\big(g+2\|g\|_\infty\big)\cdot
		\Big(\int_{\qu\Omega}|g+2\|g\|_\infty|\,d|\sigma|\Big)^{-1}\cdot\|\sigma\|
		\,.
	\]
	Then $h\in\mc M_\sigma$, and is not equal to a constant $\tilde\sigma$-a.e. 

	Let us show that $\|h\|_\infty>1$. We argue by contradiction. If we had
	$\|h\|_\infty\leq 1$, then $1-h\geq 0$ $\tilde\sigma$-a.e.\ (and hence also
	$|\sigma|$-a.e.). This implies 
	\[
		\int_{\qu\Omega}|1-h|\,d|\sigma|=\int_{\qu\Omega}(1-h)\,d|\sigma|=
		\|\sigma\|-\int_{\qu\Omega}h\,d|\sigma|=0
		\,,
	\]
	and hence $h=1$ $|\sigma|$-a.e., a contradiction. 

	Set $t:=\frac 1{\|h\|_\infty}$, then $t\in(0,1)$. Consider the function 
	\[
		\tilde h:=\frac{1-th}{1-t}
		\,.
	\]
	Clearly, $\tilde h\in\big(\frac{d\sigma}{d|\sigma|}\big)^{-1}\mc L^\perp$. Since 
	$t<1$, we have $\tilde h\geq 0$. Moreover, 
	\[
		\int_{\qu\Omega}\tilde h\,d|\sigma|=\frac 1{1-t}\Big(\|\sigma\|-
		t\int_{\qu\Omega}h\,d|\sigma|\Big)=\|\sigma\|
		\,.
	\]
	Together, we see that $\tilde h\in\mc M_\sigma$. Writing 
	\[
		1=t\cdot h+(1-t)\cdot\frac{1-th}{1-t}
	\]
	shows that $1$ is not an extremal point of $\mc M_\sigma$. 
\end{proof}

\noindent
Finally, let us provide the required facts from \cite{pitt:1983}. 
Again for completeness, we show how they are extracted from this paper. 

\begin{theorem}[Pitt]\thlab{N19}
	Let $\mu$ be a positive Borel measure on the real line. Let $\mc L$ be an algebraic 
	de~Branges space with 
	\begin{enumerate}[$(i)$]
		\item $\displaystyle \int_{\bb R}|F|\,d\mu<\infty,\quad F\in\mc L$;
		\item If $F$ satisfies $\int_{\bb R}|F|\,d\mu=0$, then $F=0$.
	\end{enumerate}
	Assume that $\mc L$ is not dense in $L^1(\mu)$. 

	Then the function $\mf m:\bb C\to[0,\infty]$ defined as 
	\[
		\mf m(z):=\sup\big\{|F(z)|:\,F\in\mc L,\|F\|_{L^1(\mu)}\leq 1\big\}
	\]
	is everywhere finite and continuous. Each element $f\in\Clos_{L^1(\mu)}\mc L$ equals 
	$\mu$-a.e.\ the restriction of an entire function $F$ with 
	\begin{equation}\label{N11}
		|F(z)|\leq \mf m(z) \|f\|,\qquad z\in\bb C
		\,.
	\end{equation}
	For each two functions $f,g\in\Clos_{L^1(\mu)}\mc L$, and entire functions $F,G$ with
	$F|_{\bb R}=f$, $G|_{\bb R}=g$\ $\mu$-a.e., which are subject to \eqref{N11}, the
	quotient $\frac FG$ is a meromorphic function of bounded type in both half--planes 
	$\bb C^+$ and $\bb C^-$. 
\end{theorem}
\begin{proof}
	Assume that there exists $z\in\bb C\setminus\bb R$ with $\mf m(z)=\infty$. By symmetry,
	also $\mf m(\qu z)=\infty$. We obtain from \cite[Proposition 2.4,Theorem 3.1]{pitt:1983}
	that $\mc L$ is dense in $L^1(\mu)$, a contradiction. Hence, $\mf m$ is finite on $\bb
	C\setminus\bb R$. By \cite[Theorem 3.2]{pitt:1983}, the function $\mf m$ is finite and
	continuous in the whole plane. The fact that each function $f\in\Clos_{L^1(\mu)}\mc L$
	can be ($\mu$-a.e.) extended to an entire function subject to \eqref{N11} is shown in the
	same theorem\footnote{In \cite[Theorem 3.2]{pitt:1983} it is claimed that this extension
	is unique. It seems that this is not true in general: the word `unique' at the end
	of the third line of this statement should be deleted.}. By 
	\cite[Proposition 3.4, Theorem A.1]{pitt:1983}, the quotient of each two such functions 
	is of bounded type in $\bb C^+$ and $\bb C^-$. 
\end{proof}

\section{Necessity of `$\bm{\protect\mathcal{K}(\mc L,W)\neq\emptyset}$'}

In this section we show the implication $(i)\Rightarrow(ii)$ in \thref{N3}. 
The proof is carried out in five steps. Throughout the discussion, we denote 
\begin{itemize}
	\item[--] by $\bb H(\bb C)$ the space of all entire functions endowed with the topology of locally uniform convergence;
	\item[--] by $\rho:\bb H(\bb C)\to C(\bb R)$ the restriction map $F\mapsto F|_{\bb R}$;
	\item[--] by $\chi_w:\bb H(\bb C)\to\bb C$ the point evaluation map $F\mapsto F(w)$.
\end{itemize}
Notice that the case that $\mc L=\{0\}$ in \thref{N3} is trivial:
we can choose for $B$ any function of the form $B(z)=z-x_0$ with $x_0\in\bb R$ such that
$W(x_0)\neq 0$. Hence, we may assume throughout that $\mc L\neq\{0\}$. 

In Steps 1 and 2, we do not use the assumption \thref{N3}, $(i)$. The arguments given in these steps work in general. 
From Step 3 on, the assumption of the theorem enters in the form of de~Branges lemma. 

%%%
%
\begin{center}
	\textit{\underline{Step 1; The bounded extension operator.}}
\end{center}
Let $\mu$ be a positive Borel measure on the real line, and let $\mc L$ be an algebraic de~Branges space 
which is injectively contained in $L^1(\mu)$ and is not dense in this space. 

By Pitt's theorem each element $f\in\qu{\rho\mc L}$ can be extended to an entire function. In this step we show, among other things, 
that one can achieve that this extension process is a linear and continuous map. 

Applying Pitt's theorem, we obtain that the function 
\[
	\mf m(z):=\sup\big\{|F(z)|:\,\|\rho F\|_{L^1(\mu)}\leq 1\big\},\quad z\in\bb C
	\,,
\]
is everywhere finite and continuous. In particular, it is locally bounded, and hence the map 
\[
	\iota:=(\rho|_{\mc L})^{-1}:\rho\mc L\subseteq L^1(\mu)\longrightarrow\mc L
\]
is continuous in the topology of $\bb H(\bb C)$. 
Denote by $\tilde\iota:\qu{\rho\mc L}\subseteq L^1(\mu)\to\bb H(\bb C)$ 
its extension by continuity. 

It is important to show that $\rho\circ\tilde\iota=\id_{\qu{\rho\mc L}}$. Note that the map
$\rho$ is in general not continuous; locally uniform convergence need not imply
$L^1$-convergence. Hence the stated equality does not follow at once, just by 'extension by
continuity'. 
Let $f\in\qu{\rho\mc L}$ be given. Choose a sequence $(F_n)_{n\in\bb N}$, 
$F_n\in\mc L$, with 
\begin{equation}\label{N12}
	\lim_{n\to\infty}\rho F_n=f\quad\text{in\ }L^1(\mu)
	\,,
\end{equation}
and extract a subsequence $(F_{n_k})_{k\in\bb N}$ such that 
\[
	\lim_{k\to\infty}(\rho F_{n_k})(x)=f(x),\quad x\in\bb R\ \text{$\mu$-a.e.}
\]
By continuity of $\tilde\iota$, the relation \eqref{N12} implies that $\lim_{n\to\infty}\tilde\iota\rho F_n=\tilde\iota f$ locally 
uniformly. In particular, 
\[
	\lim_{k\to\infty}\big(\rho\tilde\iota\rho F_{n_k}\big)(x)=(\rho\tilde\iota f)(x),\quad x\in\bb R
	\,.
\]
However, by the definition of $\tilde\iota$ we have $\rho\circ\tilde\iota|_{\rho\mc L}=
\id_{\rho\mc L}$, and hence 
\[
	\rho\tilde\iota\rho F_{n_k}=\rho F_{n_k}
	\,.
\]
We conclude that $f(x)=(\rho\tilde\iota f)(x)$, $x\in\bb R$\ $\mu$-a.e., and this means that $f=\rho\tilde\iota f$ in $L^1(\mu)$. 

Setting 
\[
	\tilde{\mc L}:=\tilde\iota(\qu{\rho\mc L}),\qquad 
	\tilde\phi_w:=\chi_w\circ\tilde\iota,\quad \phi_w:=\chi_w\circ\iota=\tilde\phi_w|_{\rho\mc L}
	\,,
\]
we can summarize in a diagram:
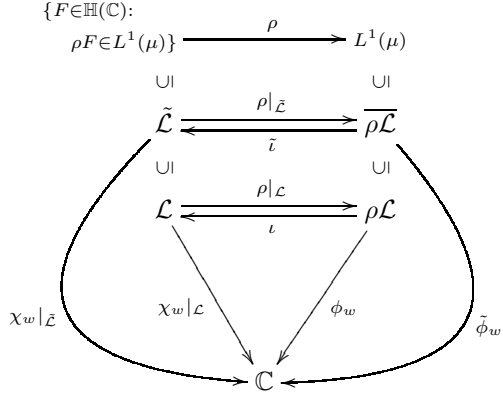
\[
\xymatrix@!R=3mm{
	\parbox{3mm}{\vspace*{-6mm}\hspace*{-19mm}
		\parbox{25mm}{\begin{multline*}\scriptstyle\{F\in\bb H(\bb C):\\[-3pt] \scriptstyle\rho F\in L^1(\mu)\}\end{multline*}}}
		\ar[rr]^\rho
		&& 
		{\scriptstyle L^1(\mu)}	
		\\
	\tilde{\mc L} \ar@<2pt>[rr]^{\rho|_{\tilde{\mc L}}} \ar@/_5pc/[dddr]_{\chi_w|_{\tilde{\mc L}}}
		\ar@{}[u]_{\begin{rotate}{90}\hspace*{-3pt}\raisebox{-2pt}{$\scriptstyle\subseteq$}\end{rotate}} 
		&& 
		\qu{\rho\mc L} \ar@<2pt>[ll]^{\tilde\iota} \ar@/^5pc/[dddl]^{\tilde\phi_w}
		\ar@{}[u]_{\begin{rotate}{90}\hspace*{-3pt}\raisebox{-2pt}{$\scriptstyle\subseteq$}\end{rotate}}
		\\
	\mc L \ar@<2pt>[rr]^{\rho|_{\mc L}} \ar[ddr]_{\chi_w|_{\mc L}}
		\ar@{}[u]_{\begin{rotate}{90}\hspace*{-3pt}\raisebox{-2pt}{$\scriptstyle\subseteq$}\end{rotate}} 
		&& 
		\rho\mc L \ar@<2pt>[ll]^{\iota} \ar[ddl]^{\phi_w}
		\ar@{}[u]_{\begin{rotate}{90}\hspace*{-3pt}\raisebox{-2pt}{$\scriptstyle\subseteq$}\end{rotate}}
		\\
	&& 
		\\
	& \bb C
}
\]
Notice that, since $\tilde{\mc L}=\ran\tilde\iota$, we also have
$\tilde\iota\circ\rho|_{\tilde{\mc L}}=\id_{\tilde{\mc L}}$. I.e., the map 
$\rho|_{\tilde{\mc L}}$ maps
$\tilde{\mc L}$ bijectively onto $\qu{\rho\mc L}$ and its inverse equals $\tilde\iota$. 
Moreover, the already noted fact that $\rho$ is in general not continuous reflects in the fact
that $\tilde{\mc L}$ is in general not closed in $\bb H(\bb C)$. For example, in the case that 
$\mc L=\bb C[z]$, the closure of $\mc L$ is all of $\bb H(\bb C)$. 

Let us compute the norm of the functionals $\tilde\phi_w$. Let $f\in\rho\mc L$ and set
$F:=\tilde\iota f$. Then 
\[
	|\tilde\phi_wf|=|(\chi_w\circ\tilde\iota)f|=|F(w)|\leq\mf m(w)\|f\|
	\,.
\]
By continuity this relation holds for all $f\in\qu{\rho\mc L}$, and we obtain that 
\[
	\|\tilde\phi_w\|\leq\mf m(w),\qquad w\in\bb C
	\,.
\]
In other words, each function $F\in\tilde{\mc L}$ satisfies 
\[
	|F(w)|\leq\mf m(w)\|\rho F\|,\qquad w\in\bb C
	\,.
\]

%%%
%
\begin{center}
	\textit{\underline{Step 2; Showing 'algebraic de~Branges space'.}}
\end{center}
Consider the same setting as in the previous step. 
We are going to show that $\tilde{\mc L}$ is an algebraic de~Branges space. 

The map ${\rule{0pt}{5pt}.}^\#:F\mapsto F^\#$ maps $\bb H(\bb C)$ continuously into itself and
is involutory. Since $\mc L$ is an algebraic de~Branges space, its restriction 
${\rule{0pt}{5pt}.}^\#|_{\mc L}$ maps $\mc L$ onto $\mc L$. Complex conjugation 
$\qu{\rule{0pt}{5pt}.}:f\mapsto\qu f$ is an involutory homeomorphism of $L^1(\mu)$ onto itself. Clearly, 
\[
	\rho|_{\mc L}\circ\big({\rule{0pt}{5pt}.}^\#|_{\mc L}\big)=
	\big(\qu{\rule{0pt}{5pt}.}|_{\rho\mc L}\big)\circ\rho|_{\mc L}
	\,.
\]
First, this relation implies that $\qu{\rule{0pt}{5pt}.}$ maps $\rho\mc L$ onto itself, and hence also $\qu{\rho\mc L}$ onto itself. 
Second, it implies that ${\rule{0pt}{5pt}.}^\#\circ\iota=\iota\circ(\qu{\rule{0pt}{5pt}.}|_{\rho\mc L})$, and hence, by continuity, 
that also ${\rule{0pt}{5pt}.}^\#\circ\tilde\iota=\tilde\iota\circ(\qu{\rule{0pt}{5pt}.}|_{\qu{\rho\mc L}})$. Thus, in fact, 
${\rule{0pt}{5pt}.}^\#$ maps $\tilde{\mc L}$ into itself. 

Let $w\in\bb C$ be fixed, and consider the difference quotient operator $\mc R_w:\bb H(\bb C)\!\times\!\bb H(\bb C)\to\bb H(\bb C)$, that is 
\[
	\mc R_w[F,G](z):=
	\begin{cases}
		\frac{F(z)G(w)-G(z)F(w)}{z-w} &\hspace*{-3mm},\quad z\neq w\\ 
		{\scriptstyle F'(w)G(w)-G'(w)F(w)} &\hspace*{-3mm},\quad z=w\\ 
	\end{cases}
	\ ,\quad F,G\in\bb H(\bb C)
	\,.
\]
By the Schwarz lemma we have, for each compact set $K$ (denote $B_1(w):=\{z\in\bb C:\,|z-w|\leq 1\}$), 
\[
	\sup_{z\in K}|\mc R_w[F,G](z)|\leq 2\sup_{z\in K\cup B_1(w)}|F(z)|\cdot\sup_{z\in K\cup B_1(w)}|G(z)|,\quad F,G\in\bb H(\bb C)
	\,,
\]
and this implies that $\mc R_w$ is continuous. 

The $L^1$-counterpart of the difference quotient operator is the map $\mc R_w^1$ 
defined on $\qu{\rho\mc L}\!\times\!\qu{\rho\mc L}$ as (the second alternative occurs of course only if $w\in\bb R$) 
\[
	\mc R_w^1[f,g](x):=
	\begin{cases}
		\frac{f(x)(\tilde\iota g)(w)-g(x)(\tilde\iota f)(w)}{x-w} &\hspace*{-3mm},\quad x\neq w\\ 
		{\scriptstyle (\tilde\iota f)'(w)(\tilde\iota g)(w)-(\tilde\iota g)'(w)(\tilde\iota f)(w)} &\hspace*{-3mm},\quad x=w\\ 
	\end{cases}
	\ ,\quad f,g\in\qu{\rho\mc L}
	\,.
\]
From this definition we immediately see that 
\begin{equation}\label{N32}
	\mc R_w^1=\rho\circ\mc R_w\circ(\tilde\iota\!\times\!\tilde\iota)
	\,.
\end{equation}
Let us show that indeed $\mc R_w^1[f,g]\in L^1(\mu)$ and that $\mc R_w^1$ is continuous. First, 
\[
	\big\|\mathds{1}_{\bb R\setminus B_1(w)}\mc R_w^1[f,g]\big\|_{L^1(\mu)}\leq 
	2\mf m(w)\|f\|_{L^1(\mu)}\|g\|_{L^1(\mu)}
	\,,
\]
and this implies that the map $(f,g)\mapsto\mathds{1}_{\bb R\setminus B_1(w)}\mc R_w^1[f,g]$ maps $\qu{\rho\mc L}\!\times\!\qu{\rho\mc L}$ 
continuously into $L^1(\mu)$. The function $\mathds{1}_{\bb R\cap B_1(w)}\mc R_w^1[f,g]$ is $\mu$-a.e.\ 
piecewise continuous and has compact support. Hence, it 
clearly belongs to $L^1(\mu)$. Moreover, due to \eqref{N32}, the map $(f,g)\mapsto\mathds{1}_{\bb R\cap B_1(w)}\mc R_w^1[f,g]$ is continuous as 
a composition of continuous maps. Note here that, although $\rho$ itself is not continuous, for each compact set $K$ the map 
$F\mapsto\mathds{1}_K\cdot\rho F$ is. 

Since $\mc L$ is an algebraic de~Branges space, we have $\mc R_w(\mc L\!\times\!\mc L)\subseteq\mc L$. Due to \eqref{N32}, thus also 
$\mc R_w^1(\rho\mc L\!\times\!\rho\mc L)\subseteq\rho\mc L$, and continuity implies 
\[
	\mc R_w^1\big(\,\qu{\rho\mc L}\!\times\!\qu{\rho\mc L}\,\big)\subseteq\qu{\rho\mc L}
	\,.
\]
Now we may multiply \eqref{N32} with $\tilde\iota$ from the left and $\rho\!\times\!\rho$ from the right to obtain 
\[
	\tilde\iota\circ\mc R_w^1\circ(\rho\!\times\!\rho)|_{\tilde{\mc L}\!\times\!\tilde{\mc L}}=\mc R_w|_{\tilde{\mc L}\!\times\!\tilde{\mc L}}
\]
and conclude that $\mc R_w(\tilde{\mc L}\!\times\!\tilde{\mc L})\subseteq\tilde{\mc L}$. In particular, $\tilde{\mc L}$ is invariant with respect to 
division of zeros. 

\begin{center}
	\textit{\underline{Step 3; Invoking de~Branges' Lemma.}}
\end{center}
From now on assume that \thref{N3}, $(i)$, holds. 
De~Branges' lemma provides us with a positive Borel measure $\mu$, $\mu\neq 0$, and a function 
$g_0\in L^\infty(\mu)$ with $|g_0|=1$\ $\mu$-a.e., such that 
\begin{enumerate}[$(i)$]
\item $\int_{\bb R}W\,d\mu<\infty$, in particular $\rho\mc L\subseteq L^1(\mu)$;
\item $(\rho\mc L)^\perp=\spn\{g_0\}$, and hence $\qu{\rho\mc L}=\{g_0\}^\perp$.
\end{enumerate}
Since $\mc L\neq\{0\}$, the support of the measure $\mu$ must contain at least two points. 

The first thing to show is that $\rho$ maps $\mc L$ injectively into $L^1(\mu)$. Assume that $F\in\mc L$ and that $F|_{\bb R}=0$\ $\mu$-a.e. 
If $\supp\mu$ is not discrete, this implies immediately that $F=0$. Hence assume that $\supp\mu$ is discrete. Then we must have $F(x)=0$, 
$x\in\supp\mu$. Pick $x_0\in\supp\mu$, denote by $l$ the multiplicity of $x_0$ as a zero of $F$, and set $G(z):=(z-x_0)^{-l}F(z)$. Then 
$G\in\mc L$, and $G(x_0)\neq 0$ whereas $G(x)=0$ for all $x\in\supp\mu\setminus\{x_0\}$. Since $g_0(x_0)\neq 0$, this contradicts 
the fact that $\int_{\bb R}Gg_0\,d\mu=0$. 

Next, we show that the measure $\mu$ is discrete. Assume on the contrary that $x_0\in\bb R$ is an accumulation point of $\supp\mu$. Choose an 
interval $[a,b]$, such that $x_0\not\in[a,b]$ and ${\rm card}\, ([a,b]\cap\supp\mu)
\geq 2$. Then $\dim  L^1(\mu|_{[a,b]})>1$, and we can choose $f\in L^1(\mu)\setminus\{0\}$ with 
\[
	f(x)=0,\ x\in\bb R\setminus[a,b],\qquad \int_{\bb R}fg_0\,d\mu=0
	\,.
\]
Then $f\in\qu{\rho\mc L}$, and hence 
\[
	f(x)=(\rho\tilde\iota f)(x),\quad x\in\bb R\ \text{$\mu$-a.e.}
\]
The function $F:=\tilde\iota f$ is entire and does not vanish identically. However, since $\rho F=f$\ $\mu$-a.e., we must have $F(x)=0$, 
$x\in\bb R\setminus[a,b]$\ $\mu$-a.e. The set $(\supp\mu)\setminus[a,b]$ has the accumulation point $x_0$, and we conclude that $F=0$, 
a contradiction. 

As a consequence, we can interpret the action of $\tilde\phi_x$ for $x\in\supp\mu$ as point 
evaluation: Let $f\in\qu{\rho\mc L}$, then $f=\rho\tilde\iota f$ in $L^1(\mu)$, i.e.,
\[
	f(x)=(\tilde\iota f)(x),\quad x\in\supp\mu
	\,.
\]
It follows that 
\[
	\tilde\phi_xf=(\chi_x\circ\tilde\iota)f=(\tilde\iota f)(x)=f(x),\quad x\in\supp\mu
	\,.
\]

%%%
%
\begin{center}
	\textit{\underline{Step 4; The functions $H_t$.}}
\end{center}
Fix a point $t_0\in\supp\mu$. For each $t\in[\supp\mu]\setminus\{t_0\}$ we define a function $h_t:\bb R\to\bb C$ as 
\[
	h_t(x):=
	\begin{cases}
		-[(t-t_0)g_0(t_0)\mu(\{t_0\})]^{-1}, &\hspace*{-3mm}\quad x=t_0,\\ 
		[(t-t_0)g_0(t)\mu(\{t\})]^{-1}, &\hspace*{-3mm}\quad x=t,\\ 
		0, &\hspace*{-3mm}\quad \text{otherwise}.\\ 
	\end{cases}
\]
Remember here that $\supp\mu$ contains at least two points. We have 
$h_t\in L^1(\mu)$ and $\int_{\bb R}h_tg_0\,d\mu=0$, and hence $h_t\in\qu{\rho\mc L}$. Moreover, 
$h_t\in\ker\tilde\phi_{t'}$ whenever $t'\in[\supp\mu]\setminus\{t_0,t\}$. Define 
\[
	H_t:=\tilde\iota h_t
	\,.
\]
We establish the essential properties of the functions $H_t$ in the following three lemmata. 

\begin{lemma}\thlab{N14}
	Let $t,t'\in[\supp\mu]\setminus\{t_0\}$. Then 
	\begin{equation}\label{N15}
		(z-t)H_t(z)=(z-t')H_{t'}(z)
		\,.
	\end{equation}
\end{lemma}
\begin{proof}
	For $t=t'$ this relation is of course trivial. Hence, assume that $t\neq t'$. Choose a function $G_0\in\mc L$ with $G_0(t')=1$, and 
	consider the function 
	\[
		f:=\big(\id+(t-t')\mc R_{t'}^1[.,\rho G_0]\big)h_t\in\qu{\rho\mc L}
		\,.
	\]
	The value of $f$ at points $x\in\bb R\setminus\{t'\}$ is computed easily from the definition of $\mc R_{t'}^1[.,\rho G_0]$: 
	\begin{multline*}
		f(x)=h_t(x)+(t-t')\frac{h_t(x)}{x-t'}=\frac{x-t}{x-t'}h_t(x)=
			\\
		=
		\begin{cases}
			-[(t'-t_0)g_0(t_0)\mu(\{t_0\})]^{-1}, &\hspace*{-3mm}\quad x=t_0,\\ 
			0, &\hspace*{-3mm}\quad x\in\bb R\setminus\{t_0,t'\},\\ 
		\end{cases}
	\end{multline*}
	Since $f\in\qu{\rho\mc L}$, we have $\int_{\bb R}fg_0\,d\mu=0$, and hence the value of 
	$f$ at $t'$ must be 
	\[
		f(t')=[(t'-t_0)g_0(t')\mu(\{t'\})]^{-1}
		\,.
	\]
	We see that $f=h_{t'}$. 

	Now we can compute 
	\begin{multline*}
		H_{t'}=\tilde\iota h_{t'}=\tilde\iota\big([\id+(t-t')\mc R_{t'}^1[.,\rho G_0]\big)h_t=
		\\
		=\big(\id+(t'-t)\mc R_{t'}[.,\rho G_0]\big)\tilde\iota h_t=\big(\id+(t'-t)\mc R_{t'}[.,\rho G_0]\big)H_t
		\,.
	\end{multline*}
	However, 
	\[
		\big(\id+(t'-t)\mc R_{t'}[.,\rho G_0]\big)H_t(z)=\frac{z-t}{z-t'}H_t(z)
		\,,
	\]
	and the desired relation \eqref{N15} follows.
\end{proof}

\begin{lemma}\thlab{N16}
	We have $H_t=H_t^\#$. The function $H_t$ has simple zeros at the points $[\supp\mu]\setminus\{t_0,t\}$, and no zeros otherwise. 
\end{lemma}
\begin{proof}
	Since $h_t=\qu{h_t}$, we have 
	\[
		H_t^\#=(\tilde\iota h_t)^\#=\tilde\iota(\qu{h_t})=\tilde\iota h_t=H_t
		\,.
	\]
	Let $t'\in[\supp\mu]\setminus\{t_0,t\}$. Since $H_{t'}(t')=h_{t'}(t')\neq 0$, the relation \eqref{N15} shows that $H_t$ has a simple zero 
	at $t'$. For $x\in\{t_0,t\}$, we have $H_t(x)=h_t(x)\neq 0$. 

	Let $w\in\bb C\setminus\supp\mu$, and assume on the contrary that $H_t(w)=0$. Then $H_t\in\ker(\chi_w|_{\tilde{\mc L}})$, and therefore 
	(choose $G_0\in\mc L$ with $G_0(w)=1$) 
	\[
		G:=\big(\id+(w-t)\mc R_w[.,G_0]\big)H_t\in\tilde{\mc L}
		\,.
	\]
	This implies that $\rho G\in\qu{\rho\mc L}$. 

	Clearly, $G(z)=\frac{z-t}{z-w}H_t(z)$, and we can evaluate $G$ at points $x\in\supp\mu$ as 
	\[
		G(x)=
		\begin{cases}
			[(t_0-w)g_0(t_0)\mu(\{t_0\})]^{-1}, &\hspace*{-3mm}\quad x=t_0,\\ 
			0, &\hspace*{-3mm}\quad x\in[\supp\mu]\setminus\{t_0\}.\\ 
		\end{cases}
	\]
	This shows that $\int_{\bb R}Gg_0\,d\mu\neq 0$, and we have reached a contradiction. 
\end{proof}

\begin{lemma}\thlab{N17}
	For each $F\in\tilde{\mc L}$ we have 
	\[
		F(z)=\sum_{\substack{t\in\supp\mu\\ t\neq t_0}} F(t)\mu(\{t\})g_0(t)(t-t_0)H_t(z)
		\,,
	\]
	where the series converges locally uniformly. 
\end{lemma}
\begin{proof}
	Since $|g_0|=1$\ $\mu$-a.e., we have 
	\[
		\|h_t\|_{L^1(\mu)}=\int_{\bb R}|h_t|\,d\mu=\frac 1{|t-t_0|}\Big(\frac 1{|g_0(t_0)|}+\frac 1{|g_0(t)|}\Big)=\frac 2{|t-t_0|}
		\,.
	\]
	Hence, 
	\[
		\|g_0(t)(t-t_0)h_t\|_{L^1(\mu)}=2
		\,,
	\]
	and therefore for each $f\in L^1(\mu)$ the series 
	\[
		g:=\sum_{\substack{t\in\supp\mu\\ t\neq t_0}} f(t)\mu(\{t\})\cdot g_0(t)(t-t_0)h_t
	\]
	converges in the norm of $L^1(\mu)$. Since $h_t\in\qu{\rho\mc L}$, it follows that also $g\in\qu{\rho\mc L}$, i.e.\ 
	$\int_{\bb R}gg_0\,d\mu=0$. 

	For $x\in[\supp\mu]\setminus\{t_0\}$, we can evaluate 
	\begin{multline*}
		g(x)=\sum_{\substack{t\in\supp\mu\\ t\neq t_0}} f(t)\mu(\{t\})\cdot g_0(t)(t-t_0)h_t(x)=
			\\
		=f(x)\mu(\{x\})\cdot g_0(x)(x-t_0)h_x(x)=f(x)
			\,.
	\end{multline*}
	Hence, the functions $g$ and $f$ differ, up to a $\mu$-zero set, at most at the point $t_0$. If we know in addition that 
	$f\in\qu{\rho\mc L}$, then also $\int_{\bb R}fg_0\,d\mu=0$, and it follows that $f(t_0)=g(t_0)$, i.e., that 
	$f=g$ in $L^1(\mu)$. 

	Now let $F\in\tilde{\mc L}$ be given. Then $\rho F\in\qu{\rho\mc L}$, and therefore 
	\[
		\rho F=\sum_{\substack{t\in\supp\mu\\ t\neq t_0}} F(t)\mu(\{t\})\cdot g_0(t)(t-t_0)h_t
		\,.
	\]
	Applying $\tilde\iota$, yields the desired representation of $F$. 
\end{proof}

%%%
%
\begin{center}
	\textit{\underline{Step 5; Construction of $B\in\mc K(\mc L,W)$.}}
\end{center}
Choose $t\in[\supp\mu]\setminus\{t_0\}$, and define 
\begin{equation}\label{N30}
	B(z):=(z-t_0)(z-t)H_t(z)
	\,.
\end{equation}
Due to \eqref{N15}, this definition does not depend on the choice of $t$. We are going to show 
that $B\in\mc K(\tilde{\mc L},W)$. 

By \thref{N17}, we have $B=B^\#$, and know that $B$ has simple zeros at the points $\supp\mu$ and no zeros otherwise; this is {\rm(K1)}. 
By \thref{N19}, for each $F\in\tilde{\mc L}$, the function $\frac F{H_t}$ is of bounded type in 
$\bb C^+$ and $\bb C^-$. Hence also $\frac FB$ has this property; and this is {\rm(K2)}. 

To show {\rm(K3)}, let $F\in\tilde{\mc L}$ be given. By \thref{N17}, 
\begin{align*}
	\frac{F(z)}{B(z)}= & \sum_{\substack{t\in\supp\mu\\ t\neq t_0}} F(t)\mu(\{t\})g_0(t)(t-t_0)\frac{H_t(z)}{B(z)}=
		\\
	= & \sum_{\substack{t\in\supp\mu\\ t\neq t_0}} F(t)\mu(\{t\})\cdot g_0(t)\frac{t-t_0}{z-t}\cdot\frac 1{z-t_0}
		\,.
\end{align*}
We have 
\[
	\sum_{t\in\supp\mu}|F(t)|\mu(\{t\})<\infty,\quad 
	|g_0(t)|=1,\quad 
	\sup_{\substack{|y|\geq 1\\ t\in\supp\mu}}\Big|\frac{t-t_0}{iy-t}\Big|<\infty
	\,,
\]
and hence, by bounded convergence, 
\[
	\lim_{y\to\pm\infty}\frac{F(iy)}{B(iy)}=0
	\,.
\]
Finally, for {\rm(K4)}, compute ($t\in[\supp\mu]\setminus\{t_0\}$ arbitrary) 
\begin{align*}
	B'(x)= & 
		\begin{cases}
			(x-t_0)H_x(x), &\hspace*{-3mm}\quad x\in[\supp\mu]\setminus\{t_0\}\\ 
			(x-t)H_t(x), &\hspace*{-3mm}\quad x=t_0\\ 
		\end{cases}
		\\
	= & \frac 1{g_0(x)\mu(\{x\})}
		\,.
\end{align*}
Remembering that $|g_0(x)|=1$\ $\mu$-a.e., we conclude that 
\[
	\sum_{x\in\supp\mu}\frac{W(x)}{|B'(x)|}=\int_{\bb R} W\,d\mu<\infty
	\,,
\]
and this is {\rm(K4)}.

\section{Computing mean type}

In this section we show the additional statement in \thref{N3}, that the function $B$ can be chosen such that 
\eqref{N27} holds. In fact, we show that the function $B$ constructed in the previous section, cf.\ \eqref{N30}, 
satisfies \eqref{N27}. 

First, let us observe that it is enough to prove that 
\begin{equation}\label{N26}
	\sup\Big\{\mt\frac F{\mf m}:\,F\in\mc L\Big\}=0
	\,.
\end{equation}
Indeed, by the definition of $B$, we have $(t\in[\supp\mu]\setminus\{t_0\})$ 
\[
	|B(iy)|=|iy-t_0|\cdot|iy-t|\cdot|H_t(iy)|\leq(y^2+t_0^2)^{\frac 12}(y^2+t^2)^{\frac 12}\cdot\mf m(iy)\|H_t\|_{L^1(\mu)}
	\,.
\]
Hence, for each $F\in\mc L$, 
\[
	\Big|\frac{F(iy)}{B(iy)}\Big|\geq\Big[(y^2+t_0^2)^{\frac 12}(y^2+t^2)^{\frac 12}\cdot\|H_t\|_{L^1(\mu)}\Big]^{-1}
	\cdot\frac{|F(iy)|}{\mf m(iy)}
	\,,
\]
and this implies that $\mt\frac FB\geq\mt \frac F{\mf m}$. By {\rm(K3)} always $\mt\frac FB\leq 0$, and it follows that 
\eqref{N26} implies \eqref{N27}. 

A proof of \eqref{N26} can be obtained from studying of the space 
\begin{align*}
	\mc K:=\Big\{F\in\bb H(\bb C):\quad & F|_{\bb R}\in L^1(\mu),\quad |F(iy)|=O\big(\mf m(iy)\big)\text{ as }|y|\to\infty,
		\\
	&\exists\,F_0\in\mc L\setminus\{0\}:\,\frac F{F_0}\text{ is of bounded type in }\bb C^+\text{ and }\bb C^-\Big\}
		\,,
\end{align*}
and the group of operators 
\[
	M_\alpha:\left\{
	\begin{array}{rcl}
		\bb H(\bb C) & \to & \bb H(\bb C)\\
		F(z) & \mapsto & e^{i\alpha z}F(z)
	\end{array}
	\right.
	,\qquad \alpha\in\bb R
	\,.
\]
The following fact is crucial. 

\begin{lemma}\thlab{N28}
	The restriction map $\rho:\mc K\to L^1(\mu)$ is injective. 
\end{lemma}

\noindent
This fact follows immediately from the last sentence in \cite[p.284, Remarks]{pitt:1983}. However, in \cite{pitt:1983} 
no explicit proofs of these remarks are given. Hence, we include a proof, sticking to what is needed for the present purpose\footnote{The general 
assertion in \cite[p. 284, Remarks]{pitt:1983} can be shown with the same argument, only using more of the machinery developed 
earlier in \cite{pitt:1983}.}. 

\begin{proofof}{\thref{N28}}
	Assume on the contrary that there exists a function $G\in\mc K\setminus\{0\}$ with $G|_{\bb R}=0$ $\mu$-a.e. Let 
	$F\in\mc L$, let $z\in\bb C^+$ with $G(z)\neq 0$, and consider the function 
	\[
		H(z,x) = \frac{F(z)G(x)-G(z)F(x)}{z-x},\quad x\in\bb R
		\,.
	\]
	Then $H(z,.)\in L^1(\mu)$, and 
	\[
		H(z,x)=-G(z)\frac{F(x)}{z-x},\quad x\in\bb R\ \mu\text{-a.e.}
	\]
	The proof of \cite[Theorem 3.3]{pitt:1983}, with the modification also used in \cite[Theorem 3.4]{pitt:1983}, shows that 
	$H(z,.)\in\qu{\rho\mc L}$. Hence, $\frac{F(x)}{z-x}\in\qu{\rho\mc L}$ whenever $F\in\mc L$, and since multiplication with 
	$\frac 1{z-x}$ is for each fixed $z\in\bb C\setminus\bb R$ a bounded operator on $L^1(\mu)$, it follows that 
	\[
		\frac{F(x)}{z-x}\in\qu{\rho\mc L},\quad F\in\tilde{\mc L},\ z\in\bb C\setminus\bb R,G(z)\neq 0
		\,.
	\]
	Consider now the function $H_t$ constructed in Step 4 of the previous section. A short computation shows that 
	for no nonreal $z$ the function $\frac{H_t(x)}{z-x}$ is annihilated by $g_0$. We have reached a contradiction. 
\end{proofof}

\begin{corollary}\thlab{N29}
	We may define a norm on $\mc K$ by 
	\[
		\|F\|_{\mc K}:=\|\rho F\|_{L^1(\mu)},\quad F\in\mc K
		\,,
	\]
	and $\mc K$ is complete with respect to this norm. The point evaluation maps $\chi_w|_{\mc K}:\mc K\to\bb C$ are 
	continuous with respect to $\|.\|_{\mc K}$. 
\end{corollary}
\begin{proof}
	Since $\rho$ is injective, the norm $\|.\|_{\mc K}$ is well-defined. Since $\mc K\supseteq\tilde{\mc L}$, and 
	$\dim\big(L^1(\mu)/\qu{\rho\mc L}\big)=1$, there are only two possibilities: Either $\rho\mc K=\qu{\rho\mc L}$, or 
	$\rho\mc K=L^1(\mu)$. In both cases, $\mc K$ is complete. 

	Since $\tilde{\mc L}$ is a closed subspace of $\mc K$ with finite codimension (and hence a complemented subspace), and 
	since the restriction $(\chi_w|_{\mc K})|_{\tilde{\mc L}}$ of the point evaluation map $\chi_w|_{\mc K}$ to $\tilde{\mc L}$ is 
	continuous, it follows that $\chi_w|_{\mc K}$ itself is continuous. 
\end{proof}

\noindent
Now bring in the family of spaces (parametrized by $\tau_+,\tau_-\leq 0$) 
\begin{align*}
	\mc K_{(\tau_+,\tau_-)}:=\Big\{F\in\mc K:\quad & |F(iy)|=O\big(e^{\tau_+y}\mf m(iy)\big)\text{ as }y\to+\infty\,,
		\\
	& |F(iy)|=O\big(e^{\tau_-|y|}\mf m(iy)\big)\text{ as }y\to-\infty\Big\}
		\,.
\end{align*}
It is obvious that 
\[
	M_\alpha\big(\mc K_{(\tau_+,\tau_-)}\big)\subseteq\mc K_{(\tau_+-\alpha,\tau_-+\alpha)},\quad \alpha\in[\tau_+,-\tau_-]
	\,.
\]
Applying this once again with $M_{-\alpha}$ and $\mc K_{(\tau_+-\alpha,\tau_-+\alpha)}$ in place of $M_\alpha$ and $\mc K_{(\tau_+,\tau_-)}$, 
it follows that in fact $M_\alpha|_{\mc K_{(\tau_+,\tau_-)}}$ is a bijection of $\mc K_{(\tau_+,\tau_-)}$ onto 
$\mc K_{(\tau_+-\alpha,\tau_-+\alpha)}$. 
Since $M_\alpha$ is isometric, there exists an extension $\hat M_{\alpha;\tau_+,\tau_-}$ to an isometric bijection of 
$\qu{\mc K_{(\tau_+,\tau_-)}}$ onto $\qu{\mc K_{(\tau_+-\alpha,\tau_-+\alpha)}}$. Since the point evaluation maps are continuous, we 
have $\hat M_{\alpha;\tau_+,\tau_-}=M_\alpha|_{\qu{\mc K_{(\tau_+,\tau_-)}}}$. 

As a consequence, we obtain that the space $\mc K_{(\tau_+,\tau_-)}$ is a closed subspace of $\mc K$. Indeed, let $F\in\qu{\mc K_{(\tau_+,\tau_-)}}$ 
be given. Then $M_{\tau_+}F,M_{-\tau_-}F\in\mc K$. Hence, 
\begin{align*}
	& \big|e^{i\tau_+(iy)}F(iy)\big|=O\big(\mf m(iy)\big)\text{ as }y\to+\infty\,,\\
	& \big|e^{-i\tau_-(iy)}F(iy)\big|=O\big(\mf m(iy)\big)\text{ as }y\to-\infty\,,
\end{align*}
and this gives $|F(iy)|=O(e^{\tau_+y}\mf m(iy))$, $y\to+\infty$, and $|F(iy)|=O(e^{\tau_-|y|}\mf m(iy))$, $y\to-\infty$. 

To finish the proof of \eqref{N26}, one more simple observation is needed. 

\begin{remark}\thlab{N25}
	Let $F\in\mc K\setminus\{0\}$, and set $\tau_F:=\mt\frac F{\mf m}$. Then $-\infty<\tau_F\leq 0$, and 
	\begin{equation}\label{N31}
		\begin{aligned}
			& F\in\mc K_{(\tau_+,0)},\quad \tau_+\in(\tau_F,0]\,,\\
			& F\not\in\mc K_{(\tau_+,0)},\quad \tau_+<\tau_F\,.
		\end{aligned}
	\end{equation}
	It is obvious that $\mt\frac F{\mf m}\leq 0$ and that the relations \eqref{N31} hold. Assume that 
	$\mt\frac F{\mf m}=-\infty$. Then $F\in\bigcap_{\tau_+\leq 0}\mc K_{(\tau_+,0)}$, and hence $M_\alpha F\in\mc K$ for all 
	$\alpha\leq 0$. The family $\{M_\alpha F:\,\alpha\leq 0\}$ is bounded with respect to the norm $\|.\|_{\mc K}$, and by 
	continuity of point evaluations thus also pointwise bounded. It follows that $F(z)=0$ for all $z\in\bb C^+$, and hence everywhere. 
\end{remark}

\begin{proofof}{\eqref{N26}}
	Assume on the contrary that 
	\[
		\tau:=\sup\Big\{\mt\frac F{\mf m}:\,F\in\mc L\Big\}<0
		\,.
	\]
	Then $\mc L \subseteq \mc K_{(\tau+\varepsilon, 0)}$ and, since
	$\mc K_{(\tau+\varepsilon, 0)}$ is closed in $\mc K$, we conclude that
	$\tilde{\mc L} \subseteq \mc K_{(\tau+\varepsilon, 0)}$.
	Fix $F\in\mc L\setminus\{0\}$. Let $\alpha\in(\tau_F,\tau_F-\tau)$, and choose $\varepsilon>0$ such that 
	$[\alpha-\varepsilon,\alpha+\varepsilon]\subseteq(\tau_F,\tau_F-\tau)$ and $\tau+\varepsilon\leq 0$. 

	The function $F$ belongs to $\mc K_{(\tau_F+\varepsilon,0)}$, and we have $\tau_F+\varepsilon\leq\alpha\leq 0$. Hence, 
	$M_\alpha F\in\mc K_{(\tau_F+\varepsilon-\alpha,\alpha)}$. Assume that 
	this function would belong to $\tilde{\mc L}$. Then it would also belong to $\mc K_{(\tau+\varepsilon,0)}$, and together 
	with what we know, thus $M_\alpha F\in\mc K_{(\tau+\varepsilon,\alpha)}$. 
	We have $\tau+\varepsilon<\tau_F-\alpha\leq-\alpha$, and hence 
	\[
		F=M_{-\alpha}(M_\alpha F)\in\mc K_{(\tau+\varepsilon+\alpha,0)}
		\,.
	\]
	However, $\tau+\varepsilon+\alpha<\tau_F$, a contradiction
        in view of \eqref{N31}. We conclude that 
	\[
		M_\alpha F\in\mc K\setminus\tilde{\mc L},\quad \alpha\in(\tau_F,\tau_F-\tau)
		\,.
	\]
	From this it immediately follows that 
	\[
		\spn\big\{M_\alpha F:\,\alpha\in(\tau_F,\tau_F-\tau)\big\}\cap\tilde{\mc L}=\{0\}
		\,.
	\]
	Clearly, the dimension of this linear span is infinite. Since $\rho$ is injective, we however have 
	$\dim(\mc K/\tilde{\mc L}) = \dim(\rho \mc K/\qu{\rho\mc L})\leq1$, and thus arrived at a contradiction. 
\end{proofof}

\section{Sufficiency of `$\bm{\protect\mathcal{K}(\mc L,W)\neq\emptyset}$'}

The proof of the implication `$(ii)\Rightarrow(i)$' in \thref{N3} is completed in the standard
way based on a Lagrange-type interpolation formula for functions in $\mc L$. 

\begin{lemma}\thlab{N6}
	Let $F\in\mc L$ and $B\in\mc K(\mc L,W)$. Assume in addition that $B$ satisfies 
	\begin{equation}\label{N7}
		y|F(iy)|=o\big(|B(iy)|\big),\quad y\to\pm\infty
		\,,
	\end{equation}
	\begin{equation}\label{N8}
		\sum_{x:\,B(x)=0}\frac{|x|\cdot W(x)}{|B'(x)|}<\infty
		\,.
	\end{equation}
	Then
	\begin{equation}\label{N9}
		zF(z)=\sum_{x:\,B(x)=0}\frac{xF(x)}{B'(x)}\frac{B(z)}{z-x}
		\,,
	\end{equation}
	where the series converges locally uniformly in $\bb C$. 
\end{lemma}
\begin{proof}
	The family $\big\{\frac{B(z)}{z-x}:\,x\in\bb R,B(x)=0\big\}$ is locally bounded. Since $F\in C_0(W)$, the series 
	$\sum_{x:\,B(x)=0}\frac{xF(x)}{B'(x)}$ converges absolutely. Together, the right side of \eqref{N9} converges locally uniformly. 

	The function 
	\[
		H(z):=\frac{zF(z)}{B(z)}-\sum_{x:\,B(x)=0}\frac{xF(x)}{B'(x)}\frac{1}{z-x}
	\]
	is entire, since the only possible poles (which are the points $x$ with $B(x)=0$) cancel. 
	It is of bounded type in both half-planes $\bb C^+$ and $\bb C^-$, and tends to $0$ along the imaginary axis. 
	Therefore, it vanishes identically. 
\end{proof}

\noindent
To ensure applicability of this fact, observe the following:

\begin{remark}\thlab{N10}
	Let $B\in\mc K(\mc L,W)$, $n,m\in\bb N$ with $n\geq m$, let $x_1,\dots,x_n$ be pairwise different real points with $B(x_i)\neq 0$, 
	and let $y_1,\dots,y_m$ be pairwise different zeros of $B$. Consider the function 
	\[
		\tilde B(z):=\frac{\prod_{i=1}^n(z-x_i)}{\prod_{j=1}^m(z-y_j)}\cdot B(z)
		\,.
	\]
	Then $\tilde B$ clearly satisfies {\rm(K1)}, {\rm(K2)}. Since $n\geq m$, also {\rm(K3)}
	holds. For each $x\in\bb R$ with $B(x)=0$, $x\ne y_j$, we have 
	\[
		\tilde B'(x)=\frac{\prod_{i=1}^n(x-x_i)}{\prod_{j=1}^m(x-y_j)}\cdot B'(x)
		\,.
	\]
	The zero sets of $B$ and $\tilde B$ differ only by finitely many points, and therefore 
	this relation shows that $\tilde B$ also satisfies {\rm(K4)}. Alltogether, 
	$\tilde B\in\mc K(\mc L,W)$. 
	
	In addition, the function $\tilde B$ satisfies 
	\begin{align*}
		& \tilde B(y_j)\neq 0,\ j=1,\dots,m\,,
			\\[1mm]
		& |y|^{n-m}|F(iy)|=o\big(|\tilde B(iy)|\big),\quad y\to\pm\infty,\ F\in\mc L\,,
			\\[1mm]
		& \sum_{x:\,\tilde B(x)=0}\frac{|x|^{n-m}W(x)}{|\tilde B'(x)|}<\infty
			\,.
	\end{align*}
\end{remark}

\begin{proofof}{\thref{N3}, `$\Leftarrow$'}
	Choose $B\in\mc K(\mc L,W)$ which satisfies \eqref{N7}, \eqref{N8}, and $B(0)\neq 0$, and consider the measure 
	\[
		\mu:=\sum_{x:\,B(x)=0}\frac 1{B'(x)}\delta_x
		\,,
	\]
	where $\delta_x$ denotes the Dirac measure supported on $\{x\}$. Since 
	\[
		\int_{\bb R}W\,d|\mu|=\sum_{x:\,B(x)=0}\frac{W(x)}{|B'(x)|}<\infty
		\,,
	\]
	the functional $\phi:f\mapsto\int_{\bb R}f\,d\mu$ belongs to $C_0(W)'$. 
	Since $C_{00}(\bb R)\subseteq C_0(W)$  and $B$ has at least one zero, $\phi$ does not
	vanish identically. 

	For each $F\in\mc L$, we apply \eqref{N9} with `$z=0$' and obtain 
	\[
		0= - \sum_{x:\,B(x)=0}\frac{xF(x)}{B'(x)}\frac{B(0)}{x}=
		-B(0)\sum_{x:\,B(x)=0}\frac{F(x)}{B'(x)}=-B(0)\phi(F|_{\bb R})
		\,.
	\]
	Thus, $\mc L$ is annihilated by $\mu$, in particular, $\mc L$ is not dense in $C_0(W)$. 
\end{proofof}

%---------
%   FINISH
%---------

{\footnotesize
\begin{flushleft}
	A.\,Baranov\\
	Department of Mathematics and Mechanics\\
	Saint Petersburg State University\\
	28, Universitetski pr.\\
	198504 Petrodvorets\\
	RUSSIA\\
	email: a.baranov@ev13934.spb.edu\\[5mm]
\end{flushleft}
\begin{flushleft}
	H.\,Woracek\\
	Institut f\"ur Analysis und Scientific Computing\\
	Technische Universit\"at Wien\\
	Wiedner Hauptstra{\ss}e.\ 8--10/101\\
	1040 Wien\\
	AUSTRIA\\
	email: harald.woracek@tuwien.ac.at\\[5mm]
\end{flushleft}
}

\end{document}